\documentclass[12pt,a4paper]{article}
\usepackage{amsmath,amsthm,amssymb, mathrsfs, tocloft}
\usepackage{hyperref}
\input xy
\xyoption{all}
\usepackage{epsfig}

\newtheorem{thm}{Theorem}[section]

\newtheorem{lem}[thm]{Lemma}

 \theoremstyle{definition}

\theoremstyle{remark}
\newtheorem{remark}[thm]{Remark}

\numberwithin{equation}{section}

\newcommand{\ben}{\begin{enumerate}}

\newcommand{\een}{\end{enumerate}}
\newcommand{\bit}{\begin{itemize}}
\newcommand{\eit}{\end{itemize}}

\begin{document}

\title
{A Symmetric Integral Identity for Bessel Functions with Applications to Integral Geometry}
\author{Yehonatan Salman \\ Weizmann Institute of Science \\ Email: salman.yehonatan@gmail.com}
\date{}
\maketitle

\begin{abstract}

In the article \cite{11} of L. Kunyansky a symmetric integral identity for Bessel functions of the first and second kind was proved in order to obtain an explicit inversion formula for the spherical mean transform where our data is given on the unit sphere in $\Bbb R^{n}$. The aim of this paper is to prove an analogous symmetric integral identity in case where our data for the spherical mean transform is given on an ellipse $E$ in $\Bbb R^{2}$. For this, we will use the recent results obtained by H.S. Cohl and H.Volkmer in \cite{7} for the expansions into eigenfunctions of Bessel functions of the first and second kind in elliptical coordinates.

\end{abstract}

\section{Introduction and Mathematical Background}

\subsection{Introduction}

\hskip0.6cm For a continuous function $f$, defined on $\Bbb R^{2}$, define the spherical mean transform, $Rf$ of $f$, by
\vskip-0.2cm
$$\hskip-8cm Rf:\Bbb R^{2}\times\Bbb R^{+}\rightarrow\Bbb R,$$
$$\hskip-5.5cm (Rf)(x, r) = \frac{1}{2\pi}\int_{-\pi}^{\pi}f\left(x + re^{i\theta}\right)d\theta$$
where $\Bbb R^{+}$ denotes the nonnegative ray $[0,\infty)$. Observe that for each point $(x, r)\in\Bbb R^{2}\times\Bbb R^{+}$ the spherical mean transform $Rf$ of $f$ evaluates the integral of $f$ on the circle with the center at $x$ and radius $r$.

The spherical mean transform was found to be an applicable tool in various research fields in mathematical physics and science (see \cite{1, 3, 5, 6, 8, 10, 12, 13, 14, 19}). Hence, in the last two decades many authors have investigated this integral transform with the aim of determining its kernel, range and its inverse transform $R^{-1}$. For the determination of the inverse transform, of course, one has to start with a well-posed problem in order to guarantee that $f$ can be uniquely recovered and such that no redundant data is available (i.e., that our problem is not over determined). For this, one has to restrict the domain of the spherical mean transform into a two dimensional surface in $\Bbb R^{2}\times\Bbb R^{+}$.

In this paper we will be dealing with the problem of finding the inverse transform of the spherical mean transform $R$, which is equivalent to the problem of recovering each function $f$ from its spherical mean transform $Rf$, in case where $Rf$ is restricted on a cylindrical surface $\Gamma\times\Bbb R^{+}$ where $\Gamma$ denotes a simple curve in $\Bbb R^{2}$. That is, one would like to reconstruct a continuous function $f$ in case where the integrals of $f$ are given over any circle with a center on $\Gamma$ and with an arbitrary radius. This reconstruction problem was found to be very important in many fields in science related to medical imaging such as photo and thermoacoustic tomography (\cite{1, 5, 10, 12, 14, 19}).

For the general $n$ dimensional case, inversion formulas for recovering a continuous function $f$ in $\Bbb R^{n}$, where $Rf$ is restricted to a cylindrical surface of the form $\Gamma\times\Bbb R^{+}$ with $\Gamma$ a hypersurface in $\Bbb R^{n}$, have been found in cases where $\Gamma$ is equal to various quadratic surfaces (\cite{2, 9, 11, 15, 16, 17, 20}). One of the classic inversion formulas for the case where $\Gamma$ is the unit sphere $\Bbb S^{n - 1}$ in $\Bbb R^{n}$ was obtained by L. Kunyansky in $\cite{11}$. The inversion formula in $\cite{11}$ relies on the following integral identity
\vskip-0.2cm
$$\hskip-3cm\int_{\Bbb S^{n - 1}}Y\left(\lambda\left|y - z\right|\right)\frac{\partial}{\partial n_{z}}J(\lambda|x - z|)ds(z)$$
\begin{equation}\hskip-3cm = \int_{\Bbb S^{n - 1}}Y\left(\lambda\left|x - z\right|\right)\frac{\partial}{\partial n_{z}}J(\lambda|y - z|)ds(z)\end{equation}
which is valid for $|x|, |y| < 1$, where $J(t) = J_{\frac{n}{2} - 1}(t) / t^{n / 2 - 1}, Y(t) = Y_{\frac{n}{2} - 1}(t) / t^{n / 2 - 1}$ and where $J_{\nu}$ and $Y_{\nu}$ denote respectively the Bessel functions of the first and second kind of order $\nu$. That is, if $\mathcal{I}(x, y)$ denotes the integral in the left hand side of equation (1.1) then $\mathcal{I}$ is a symmetric function in the variables $x$ and $y$, i.e., $\mathcal{I}(x, y) = \mathcal{I}(y, x)$. Restricting our discussion to the plane $\Bbb R^{2}$, identity (1.1) was proved in \cite{11} by using the following Fourier expansions
\begin{equation}\hskip-3.5cm J_{0}\left(\lambda\left|re^{i\phi} - Re^{i\theta}\right|\right) = \sum_{n = -\infty}^{\infty}J_{n}(\lambda r)J_{n}(\lambda R)e^{in(\theta - \phi)},\end{equation}
\begin{equation}\hskip-2cm Y_{0}\left(\lambda\left|\rho e^{i\psi} - Re^{i\theta}\right|\right) = \sum_{n = -\infty}^{\infty}J_{n}(\lambda \rho)Y_{n}(\lambda R)e^{in(\theta - \psi)}, \rho < R,\end{equation}
(see \cite[page 23, formula 4.55]{18}). Let $x = re^{i\phi}$ and $y = \rho e^{i\psi}$ be the polar representations of the points $x$ and $y$, then using identities (1.2) and (1.3) we obtain, by a straightforward computation, the following representation of $\mathcal{I}$:
$$\hskip-1.5cm \mathcal{I}(x, y) = \mathcal{I}\left(re^{i\phi}, \rho e^{i\psi}\right) = \lambda\sum_{n = -\infty}^{\infty}J_{n}(\lambda r)J_{n}(\lambda \rho)Y_{n}(\lambda)J_{n}'(\lambda)e^{in(\psi - \phi)}$$
\begin{equation} = \lambda J_{0}(\lambda r)J_{0}(\lambda\rho)Y_{0}(\lambda)J_{0}'(\lambda) + 2\lambda\sum_{n = 1}^{\infty}J_{n}(\lambda r)J_{n}(\lambda \rho)Y_{n}(\lambda)J_{n}'(\lambda)\cos(n(\psi - \phi))\end{equation}
which is valid for $\rho, r < 1$. Thus, since $\mathcal{I}$ is symmetric with respect to the variables $r$ and $\rho$ and with respect to the variables $\psi$ and $\phi$ it follows that it is symmetric with respect to the variables $x$ and $y$.

Our main aim in this paper is to obtain an inversion formula for the spherical mean transform in case where the centers of the circles of integration are given on an ellipse $E$ in $\Bbb R^{2}$, rather than on a circle, by using proper modifications in the methods introduced in $\cite{11}$. For this, we will have to find analogous expansions into eigenfunctions for the Bessel functions of the first and second kind, like in (1.2) and (1.3), but in elliptical coordinates in order to obtain a similar representation, as in (1.4), for the function $\mathcal{I}$. This will allow us to prove a similar integral identity to (1.1) where now integration will be over an ellipse rather than on a circle (the case $n = 2$ in (1.1)). We will rely on recent results and methods introduced by H. S. Cohl and H. Volkmer in \cite{7} in order to obtain the corresponding expansions for the elliptical case.

Before formulating our main results, we will have to introduce some mathematical notations, definitions and special functions which will be used during the next sections in the text.\\

\subsection{Mathematical Background}

Denote by $\Bbb R^{2}$ the two dimensional Euclidean plane and by $\Bbb R^{+}$ the ray $[0, \infty)$. For any integer $n$ denote by $J_{n}$ and $Y_{n}$ respectively the Bessel functions of the first and second kind of order $n$.

For a fixed point $\xi_{0}$ in $\Bbb R^{+}$, denote by $\mathrm{E}_{\xi_{0}}$ the following ellipse
\begin{equation}\hskip-4cm \mathrm{E}_{\xi_{0}} = \left\{(x_{1},x_{2})\in\Bbb R^{2}:\frac{x_{1}^{2}}{\cosh^{2}\xi_{0}} + \frac{x_{2}^{2}}{\sinh^{2}\xi_{0}} = 1\right\}\end{equation}
in $\Bbb R^{2}$. For a point $x\in\Bbb R^{2}$ denote by $x(\xi,\eta)$ its representation in elliptical coordinates:
\begin{equation}x(\xi, \eta) = (\cosh(\xi)\cos(\eta), \sinh(\xi)\sin(\eta)), (\xi, \eta)\in\Bbb R^{2}.\end{equation}
Observe that $x(\xi, \eta) = x(\xi, \eta + 2m\pi)$ for $m\in\Bbb Z$ and $x(\xi, \eta) = x(-\xi, -\eta)$. Thus, equation (1.6) does not represent the point $x$ in a unique way. We can get a bijective map in equation (1.6) by restricting the pair of variables $(\xi, \eta)$ to the domain $(0,\infty)\times[-\pi,\pi)$ which will be necessary, for example, when using the method of change of variables when performing integration. However, in general we will not assume that this is the case.

In elliptical coordinates the Helmholtz equation
\vskip-0.2cm
$$\hskip-7cm(\triangle_{x} + k^{2})U(x) = 0$$
is given by
\vskip-0.2cm
$$\hskip-2.5cm\left(\frac{\partial^{2}}{\partial\xi^{2}} + \frac{\partial^{2}}{\partial\eta^{2}} + k^{2}(\cosh^{2}\xi - \cos^{2}\eta)\right)u(\xi, \eta) = 0$$
where $u(\xi, \eta) = U\left(x(\xi, \eta)\right).$ Using separation of variables $u(\xi, \eta) = u_{1}(\xi)u_{2}(\eta)$ yields the following two equations
\vskip-0.2cm
\begin{equation}\hskip-4cm - u_{1}^{''}(\xi) + \left(\lambda - \frac{k^{2}}{2}\cosh(2\xi)\right)u_{1}(\xi) = 0,\end{equation}
\begin{equation}\hskip-4cm u_{2}^{''}(\eta) + \left(\lambda - \frac{k^{2}}{2}\cos(2\eta)\right)u_{2}(\eta) = 0.\end{equation}
The complete system of eigenfunctions to equation (1.8) is given by
$$\hskip-3cm\Theta = \left\{\textrm{ce}_{n}\left(\eta, \frac{k^{2}}{4}\right)\right\}_{n = 0}^{\infty}\bigcup\left\{\textrm{se}_{n}\left(\eta, \frac{k^{2}}{4}\right)\right\}_{n = 1}^{\infty}$$
with corresponding eigenvalues $\lambda = a_{n}(k), n\geq0$ and $\lambda = b_{n}(k), n\geq 1$. The functions $\textrm{ce}_{n}$ and $\textrm{se}_{n}$ are $2\pi$ periodic with respect to the variable $\eta$ and are respectively even and odd with respect to $\eta$ (\cite[Section 4]{7}). The family of functions $\Theta$ is orthogonal with respect to integration on the interval $[-\pi,\pi)$ in the variable $\eta$ and it can be uniquely determined by assuming the following conditions
\vskip-0.2cm
$$\hskip-6.5cm \int_{-\pi}^{\pi}\textrm{ce}_{n}^{2}\left(\eta, \frac{k^{2}}{4}\right)d\eta = \pi, n\geq0,$$
$$\hskip-6.5cm \int_{-\pi}^{\pi}\textrm{se}_{n}^{2}\left(\eta, \frac{k^{2}}{4}\right)d\eta = \pi, n\geq1.$$
The determination of the sign of $\textrm{ce}_{n}$ and $\textrm{se}_{n}$ is solved by choosing
$$\hskip-2.5cm\textrm{ce}_{0}(z,0) = \frac{1}{\sqrt{2}}, \textrm{ce}_{n}(z,0) = \cos(nz), \textrm{se}_{n}(z,0) = \sin(nz), n\geq1$$
in case where $k = 0$ and for other values of $k$ the sign is obtained by the continuity of $\textrm{ce}_{n}, \textrm{se}_{n}$ with respect to $k$ together with their values at $k = 0$.

If $f(z)$ is a $2\pi$ periodic complex function which is analytic in an open strip $S$ containing the real axis then $f$ has the following expansion
\begin{equation}f(z) = \alpha_{0}\textrm{ce}_{0}\left(z, \frac{k^{2}}{4}\right) + \sum_{n = 1}^{\infty}\left(\alpha_{n}\textrm{ce}_{n}\left(z, \frac{k^{2}}{4}\right) + \beta_{n}\textrm{se}_{n}\left(z, \frac{k^{2}}{4}\right)\right)\end{equation}
where
\vskip-0.2cm
$$\hskip-0.5cm\alpha_{n} = \frac{1}{\pi}\int_{-\pi}^{\pi}f(\eta)\textrm{ce}_{n}\left(\eta, \frac{k^{2}}{4}\right)d\eta, \beta_{n} = \frac{1}{\pi}\int_{-\pi}^{\pi}f(\eta)\textrm{se}_{n}\left(\eta, \frac{k^{2}}{4}\right)d\eta$$
and where the series (1.9) converges absolutely and uniformly on any compact subset of $S$ (see \cite[page 664, formula 28.11.1]{4}).

For equation (1.7) we will introduce the following solutions (\cite[page 667, formulas 28.20.15-16]{4})
$$\hskip-2.7cm\textrm{Mc}_{n}^{(3)}\left(\xi, \frac{k}{2}\right), n = 0, 1, 2,...,\hskip0.3cm \textrm{Ms}_{n}^{(3)}\left(\xi, \frac{k}{2}\right), n = 1, 2,...$$
which have the following asymptotic behavior (\cite[page 667, formula 28.20.9]{4})
\begin{equation}\hskip-0.55cm\textrm{Mc}_{n}^{(3)}\left(z, \frac{k}{2}\right), \textrm{Ms}_{n}^{(3)}\left(z, \frac{k}{2}\right) \approx H_{n}\left(k\cosh z\right)\left(1 + O(\textrm{sech}\hskip0.05cm z)\right), \Re z \rightarrow \infty, \left|\Im(z)\right|\leq\frac{\pi}{2}\end{equation}
where $H_{n}$ is the Hankel function of order $n$. We will also introduce the following solutions (\cite[page 667, formulas 28.20.15-16]{4})
$$\hskip-2.6cm\textrm{Mc}_{n}^{(1)}\left(\xi, \frac{k}{2}\right), n = 0, 1, 2,...,\hskip0.3cm \textrm{Ms}_{n}^{(1)}\left(\xi, \frac{k}{2}\right), n = 1, 2,...$$
which are even and odd respectively with respect to the variable $\xi$ and which have the following asymptotic behavior (\cite[page 667, formula 28.20.11]{4})
$$\hskip-9cm\textrm{Mc}_{n}^{(1)}\left(z, \frac{k}{2}\right), \textrm{Ms}_{n}^{(1)}\left(z, \frac{k}{2}\right)$$
\begin{equation} \hskip-2.2cm\approx J_{n}\left(k\cosh z\right) + e^{\left|\Im\left(k\cosh z\right)\right|}O\left(\left(\cosh z\right)^{-\frac{3}{2}}\right), \Re z \rightarrow \infty, \left|\Im(z)\right|\leq\frac{\pi}{2}.\end{equation}

For every $k > 0$ define the following functions
$$\hskip-7cm u_{k}, v_{k}:\Bbb R^{2}\times\Bbb R^{2}\rightarrow\Bbb R$$
by
\begin{equation}\hskip-3.5cm u_{k}(\xi, \eta, \xi', \eta') = J_{0}\left(k\left|x(\xi, \eta) - x(\xi', \eta')\right|\right),\end{equation}
\begin{equation}\hskip-3.5cm v_{k}(\xi, \eta, \xi', \eta') = Y_{0}\left(k\left|x(\xi, \eta) - x(\xi', \eta')\right|\right).\end{equation}

\section{Main Results}

The first main result of this paper, Theorem 2.1, is an explicit inversion formula of the spherical mean transform where our data is given on the cylindrical surface $\mathrm{E}_{\xi_{0}}\times\Bbb R^{+}$. Observe that if the spherical mean transform $Rf$ of a function $f$ is given on the cylindrical surface $\mathrm{E}_{\xi_{0}}\times\Bbb R^{+}$ then this is equivalent that we have the following data
\vskip-0.2cm
$$\hskip-4cm F(\eta, r) = (Rf)(x(\xi_{0}, \eta), r), (\eta, r)\in[-\pi,\pi]\times[0,\infty)$$
and we have to express $f$ via $F$. In Theorem 2.1 we show how to recover a continuous function $f$ from its spherical mean transform in case where the circles of integration have centers on the ellipse $\mathrm{E}_{\xi_{0}}$. The idea behind the proof of Theorem 2.1 follows the same method which was introduced in $\cite{11}$ where we use an analogous symmetric integral identity to (1.1) where integration in our case will be over the ellipse $\mathrm{E}_{\xi_{0}}$. This symmetric integral identity is an important result by itself and hence will be considered as the second main result of this paper. The exact formulation of the second main result is given in Theorem 2.2.

\begin{thm}

Let $\xi_{0} > 0$ and let $f$ be a continuous function compactly supported inside the ellipse $\mathrm{E}_{\xi_{0}}$. Then, for any $x$ inside $\mathrm{E}_{\xi_{0}}$ or equivalently for any point $x(\xi, \eta)$ such that $0 < \xi < \xi_{0}$ we have

$$\hskip-12cm f(x(\xi,\eta))$$
$$\hskip-2.5cm = \frac{1}{4}\int_{0}^{\infty}\int_{-\pi}^{\pi}\partial_{3}u_{k}(\xi, \eta, \xi_{0}, \eta'')\int_{0}^{\infty}(Rf)(x(\xi_{0}, \eta''), r)Y_{0}\left(kr\right) rdr d\eta''kdk$$
$$\hskip-2.4cm - \frac{1}{4}\int_{0}^{\infty}\int_{-\pi}^{\pi}\partial_{3}v_{k}(\xi, \eta, \xi_{0}, \eta'')\int_{0}^{\infty}(Rf)(x(\xi_{0}, \eta''), r)J_{0}\left(kr\right) rdr d\eta''kdk.$$

\end{thm}

\begin{thm}

Let $\xi, \xi'$ and $\xi_{0}$ be three real numbers satisfying $0 < \xi, \xi' < \xi_{0}$. Then, the following integral
\vskip-0.2cm
\begin{equation}\hskip-4cm \mathcal{I}(\xi, \eta, \xi', \eta') = \int_{-\pi}^{\pi}u_{k}(\xi, \eta, \xi_{0}, \eta'')\partial_{3}v_{k}(\xi', \eta', \xi_{0}, \eta'')d\eta''\end{equation}
is a symmetric function in the variables $(\xi, \eta)$ and $(\xi', \eta')$. That is, $\mathcal{I}(\xi, \eta, \xi', \eta') = \mathcal{I}(\xi', \eta', \xi, \eta)$.

\end{thm}

\begin{remark}

Observe that the integral identity (2.1) is analogous to the integral identity (1.1) where now integration is over the ellipse $\mathrm{E}_{\xi_{0}}$ where the infinitesimal length measure $ds(z)$ is replaced by $d\eta''$.

\end{remark}

\section{Proofs of the Main Results}

\textbf{Proof of Theorem 2.1}: Let $\xi_{0}$ be a fixed positive real number and let $f$ be a continuous function with compact support inside the ellipse $\mathrm{E}_{\xi_{0}}$. Define the following function
\vskip-0.2cm
\begin{equation}\hskip-3cm(Gf)(x, k) = \int_{\overset{\circ}{\textrm{E}}_{\xi_{0}}}f(y)J_{0}(k\left|x - y\right|)dy, (x, k)\in \overset{\circ}{\textrm{E}}_{\xi_{0}}\times\Bbb R^{+}\end{equation}
where $\overset{\circ}{\textrm{E}}_{\xi_{0}}$ denotes the interior of $\textrm{E}_{\xi_{0}}$ and our aim is to recover $Gf$ from the spherical mean transform $Rf$. Using Lemma 4.3 we will immediately be able to recover the function $f$.

Let $x = x(\xi, \eta)$ be the representation of $x$ in elliptical coordinates and let us make the following change of variables $y = x(\xi', \eta')$ in (3.1) to obtain
\vskip-0.2cm
$$\hskip-10.5cm (Gf)(x(\xi, \eta), k)$$
$$\hskip-0.5cm = \int_{0}^{\xi_{0}}\int_{-\pi}^{\pi}f\left(x(\xi', \eta')\right)J_{0}(k\left|x(\xi, \eta) - x(\xi', \eta')\right|)\left(\cosh^{2}\xi' - \cos^{2}\eta'\right)d\xi'd\eta'$$
\begin{equation}\hskip-2.2cm = \int_{0}^{\xi_{0}}\int_{-\pi}^{\pi}f\left(x(\xi', \eta')\right)u_{k}\left(\xi, \eta, \xi', \eta'\right)\left(\cosh^{2}\xi' - \cos^{2}\eta'\right)d\xi'd\eta'.\end{equation}
Since $\xi, \xi' < \xi_{0}$ we have, by Lemma 4.1, that
\vskip-0.2cm
$$\hskip-8cm u_{k}\left(\xi, \eta, \xi', \eta'\right) = u_{k}\left(\xi', \eta', \xi, \eta\right)$$
$$ \hskip0.5cm = \frac{1}{4}\int_{-\pi}^{\pi}\left(\partial_{3}u_{k}(\xi, \eta, \xi_{0}, \eta'')v_{k}(\xi', \eta', \xi_{0}, \eta'') - u_{k}(\xi, \eta, \xi_{0}, \eta'')\partial_{3}v_{k}(\xi', \eta', \xi_{0}, \eta'')\right)d\eta''.$$
Inserting the last identity into equation (3.2) we obtain that
$$\hskip-11cm(Gf)(x(\xi, \eta), k)$$
$$ = \frac{1}{4}\int_{-\pi}^{\pi}\partial_{3}u_{k}(\xi, \eta, \xi_{0}, \eta'')\int_{0}^{\xi_{0}}\int_{-\pi}^{\pi}f\left(x(\xi', \eta')\right)v_{k}(\xi', \eta', \xi_{0}, \eta'')\left(\cosh^{2}\xi' - \cos^{2}\eta'\right)d\xi'd\eta'd\eta''$$
$$ - \frac{1}{4}\int_{-\pi}^{\pi}u_{k}(\xi, \eta, \xi_{0}, \eta'')\int_{0}^{\xi_{0}}\int_{-\pi}^{\pi}f\left(x(\xi', \eta')\right)\partial_{3}v_{k}(\xi', \eta', \xi_{0}, \eta'')\left(\cosh^{2}\xi' - \cos^{2}\eta'\right)d\xi'd\eta'd\eta''$$
\begin{equation}\hskip-12cm = \mathcal{I}_{1} + \mathcal{I}_{2}.\end{equation}
By changing the order of integration in the integral $\mathcal{I}_{2}$, using the definition of the integral $\mathcal{I}$ in equation (2.1) and the fact that $\xi, \xi' < \xi_{0}$ we have
$$\hskip-2.5cm\mathcal{I}_{2} = -\frac{1}{4}\int_{0}^{\xi_{0}}\int_{-\pi}^{\pi}f\left(x(\xi', \eta')\right)\mathcal{I}(\xi, \eta, \xi', \eta')\left(\cosh^{2}\xi' - \cos^{2}\eta'\right)d\xi'd\eta'$$
$$\hskip-2cm = -\frac{1}{4}\int_{0}^{\xi_{0}}\int_{-\pi}^{\pi}f\left(x(\xi', \eta')\right)\mathcal{I}(\xi', \eta', \xi, \eta)\left(\cosh^{2}\xi' - \cos^{2}\eta'\right)d\xi'd\eta'$$
$$\hskip-8.8cm = - \frac{1}{4}\int_{-\pi}^{\pi}\partial_{3}v_{k}(\xi, \eta, \xi_{0}, \eta'')$$
$$\hskip-1cm\times\int_{0}^{\xi_{0}}\int_{-\pi}^{\pi}f\left(x(\xi', \eta')\right)u_{k}(\xi', \eta', \xi_{0}, \eta'')\left(\cosh^{2}\xi' - \cos^{2}\eta'\right)d\xi'd\eta'd\eta''$$
where in the first passage we used Theorem 2.2 which guarantees that the integral $\mathcal{I}$ is a symmetric function in the variables $(\xi, \eta)$ and $(\xi', \eta')$. Hence, returning to equation (3.3) we obtain that
\vskip-0.2cm
$$\hskip-11cm(Gf)(x(\xi, \eta), k)$$
$$ = \frac{1}{4}\int_{-\pi}^{\pi}\partial_{3}u_{k}(\xi, \eta, \xi_{0}, \eta'')\int_{0}^{\xi_{0}}\int_{-\pi}^{\pi}f\left(x(\xi', \eta')\right)v_{k}(\xi', \eta', \xi_{0}, \eta'')\left(\cosh^{2}\xi' - \cos^{2}\eta'\right)d\xi'd\eta'd\eta''$$
$$ - \frac{1}{4}\int_{-\pi}^{\pi}\partial_{3}v_{k}(\xi, \eta, \xi_{0}, \eta'')\int_{0}^{\xi_{0}}\int_{-\pi}^{\pi}f\left(x(\xi', \eta')\right)u_{k}(\xi', \eta', \xi_{0}, \eta'')\left(\cosh^{2}\xi' - \cos^{2}\eta'\right)d\xi'd\eta'd\eta''$$
$$ \hskip-11cm = \frac{1}{4}\int_{-\pi}^{\pi}\partial_{3}u_{k}(\xi, \eta, \xi_{0}, \eta'')$$
$$\hskip-1cm\times\int_{0}^{\xi_{0}}\int_{-\pi}^{\pi}f\left(x(\xi', \eta')\right)Y_{0}\left(k\left|x(\xi', \eta') - x(\xi_{0}, \eta'')\right|\right)\left(\cosh^{2}\xi' - \cos^{2}\eta'\right)d\xi'd\eta'd\eta''$$
$$ \hskip-11cm - \frac{1}{4}\int_{-\pi}^{\pi}\partial_{3}v_{k}(\xi, \eta, \xi_{0}, \eta'')$$
$$\hskip-1cm\times\int_{0}^{\xi_{0}}\int_{-\pi}^{\pi}f\left(x(\xi', \eta')\right)J_{0}\left(k\left|x(\xi', \eta') - x(\xi_{0}, \eta'')\right|\right)\left(\cosh^{2}\xi' - \cos^{2}\eta'\right)d\xi'd\eta'd\eta''$$
$$\hskip-6cm = \left[x = x(\xi', \eta'), dx = \left(\cosh^{2}\xi' - \cos^{2}\eta'\right)d\xi'd\eta'\right]$$
$$\hskip-4cm = \frac{1}{4}\int_{-\pi}^{\pi}\partial_{3}u_{k}(\xi, \eta, \xi_{0}, \eta'')\int_{\Bbb R^{2}}f(x)Y_{0}\left(k\left|x - x(\xi_{0}, \eta'')\right|\right) dxd\eta''$$
$$\hskip-4cm  - \frac{1}{4}\int_{-\pi}^{\pi}\partial_{3}v_{k}(\xi, \eta, \xi_{0}, \eta'')\int_{\Bbb R^{2}}f(x)J_{0}\left(k\left|x - x(\xi_{0}, \eta'')\right|\right) dxd\eta''$$
$$\hskip-8cm = \left[x = x(\xi_{0}, \eta'') + re^{i\theta}, dx = rd\theta dr\right]$$
$$\hskip-2cm = \frac{1}{4}\int_{-\pi}^{\pi}\partial_{3}u_{k}(\xi, \eta, \xi_{0}, \eta'')\int_{0}^{\infty}\int_{-\pi}^{\pi}f\left(x(\xi_{0}, \eta'') + re^{i\theta}\right)Y_{0}\left(kr\right) rdrd\theta d\eta''$$
$$\hskip-2cm - \frac{1}{4}\int_{-\pi}^{\pi}\partial_{3}v_{k}(\xi, \eta, \xi_{0}, \eta'')\int_{0}^{\infty}\int_{-\pi}^{\pi}f\left(x(\xi_{0}, \eta'') + re^{i\theta}\right)J_{0}\left(kr\right) rdrd\theta d\eta''$$
$$\hskip-3.5cm = \frac{\pi}{2}\int_{-\pi}^{\pi}\partial_{3}u_{k}(\xi, \eta, \xi_{0}, \eta'')\int_{0}^{\infty}(Rf)(x(\xi_{0}, \eta''), r)Y_{0}\left(kr\right) rdr d\eta''$$
$$\hskip-3.4cm - \frac{\pi}{2}\int_{-\pi}^{\pi}\partial_{3}v_{k}(\xi, \eta, \xi_{0}, \eta'')\int_{0}^{\infty}(Rf)(x(\xi_{0}, \eta''), r)J_{0}\left(kr\right) rdr d\eta''.$$
Now, using Lemma 4.3 and the definition (3.1) of the function $Gf$ we immediately obtain Theorem 2.1.\\\\
\textbf{Proof of Theorem 2.2}: To simplify notation during the proof of Theorem 2.2 we will use the notation $\textrm{se}_{0}$ where we define $\textrm{se}_{0}\equiv0$. From Lemma 4.2 it follows that $v_{k}$ has the following expansion
\vskip-0.2cm
$$\hskip-11cm v_{k}(\xi', \eta', \xi_{0}, \eta'')$$
$$ \hskip-1.7cm = 2\pi i\sum_{n = 0}^{\infty}\left[\textrm{Mc}_{n}^{(1)}\left(\xi', \frac{k}{2}\right)\textrm{ce}_{n}\left(\eta', \frac{k^{2}}{4}\right)\textrm{Mc}_{n}^{(3)}\left(\xi_{0}, \frac{k}{2}\right)\textrm{ce}_{n}\left(\eta'', \frac{k^{2}}{4}\right)\right.$$
\begin{equation} \hskip-0.4cm \left. + \textrm{Ms}_{n}^{(1)}\left(\xi', \frac{k}{2}\right)\textrm{se}_{n}\left(\eta', \frac{k^{2}}{4}\right)\textrm{Ms}_{n}^{(3)}\left(\xi_{0}, \frac{k}{2}\right)\textrm{se}_{n}\left(\eta'', \frac{k^{2}}{4}\right)\right]\end{equation}
which is valid in case where $|\xi'| < \xi_{0}$. Also, it was proved in \cite[Theorem 4.2]{7} that $u_{k}$ has the following expansion
\vskip-0.2cm
$$\hskip-11cm u_{k}(\xi, \eta, \xi_{0}, \eta'')$$
$$ \hskip-0.8cm = \frac{1}{\pi}\sum_{n = 0}^{\infty}\left[\mu_{n}\left(\frac{k^{2}}{4}\right)\textrm{ce}_{n}\left(i\xi, \frac{k^{2}}{4}\right)\textrm{ce}_{n}\left(\eta, \frac{k^{2}}{4}\right)\textrm{ce}_{n}\left(i\xi_{0}, \frac{k^{2}}{4}\right)\textrm{ce}_{n}\left(\eta'', \frac{k^{2}}{4}\right)\right.$$
\begin{equation}\hskip0.6cm \left. + \nu_{n}\left(\frac{k^{2}}{4}\right)\textrm{se}_{n}\left(i\xi, \frac{k^{2}}{4}\right)\textrm{se}_{n}\left(\eta, \frac{k^{2}}{4}\right)\textrm{se}_{n}\left(i\xi_{0}, \frac{k^{2}}{4}\right)\textrm{se}_{n}\left(\eta'', \frac{k^{2}}{4}\right)\right]\end{equation}
where $\mu_{n}$ and $\nu_{n}$ are functions which depend only on the variable $k$. Observe that the expansions (3.4) and (3.5) are analogous to the expansions (1.2) and (1.3) of the Bessel functions of the first and second kind of order zero where in (3.4) and (3.5) we use expansions into eigenfunctions in elliptical coordinates rather than expansions into Fourier series. In \cite{4} the following identities were shown
$$\hskip-4.5cm\textrm{Mc}_{n}^{(1)}\left(\xi', \frac{k}{2}\right) = \rho_{n,1}(k)\textrm{ce}_{n}\left(i\xi', \frac{k^{2}}{4}\right),$$
$$\hskip-4.5cm\textrm{Ms}_{n}^{(1)}\left(\xi', \frac{k}{2}\right) = \rho_{n,2}(k)\textrm{se}_{n}\left(i\xi', \frac{k^{2}}{4}\right)$$
where the functions $\rho_{n,i}, i = 1,2$ are given explicitly in \cite[page 669, formulas 28.22.1-2]{4}. Now, differentiating $v_{k}$ with respect to $\xi_{0}$, multiplying it with $u_{k}$ and then integrating on $\eta''\in[-\pi,\pi]$ we obtain, after using the orthogonality of the system $\{\textrm{ce}_{n}\}_{n = 0}^{\infty}\cup\{\textrm{se}_{n}\}_{n = 1}^{\infty}$, that
\vskip-0.2cm
$$\hskip-12.2cm \mathcal{I}(\xi, \eta, \xi', \eta') $$
$$ \hskip0.3cm = 2\pi i\sum_{n = 0}^{\infty}\left[\mu_{n}\left(\frac{k^{2}}{4}\right)\rho_{n, 1}(k)\textrm{ce}_{n}\left(i\xi, \frac{k^{2}}{4}\right)\textrm{ce}_{n}\left(i\xi', \frac{k^{2}}{4}\right)\textrm{ce}_{n}\left(\eta, \frac{k^{2}}{4}\right)\textrm{ce}_{n}\left(\eta', \frac{k^{2}}{4}\right)\right.$$
$$\hskip0.6cm\times \textrm{ce}_{n}\left(i\xi_{0}, \frac{k^{2}}{4}\right)\partial_{\xi_{0}}\textrm{Mc}_{n}^{(3)}\left(\xi_{0}, \frac{k}{2}\right)$$
$$\hskip1.5cm + \nu_{n}\left(\frac{k^{2}}{4}\right)\rho_{n, 2}(k)\textrm{se}_{n}\left(i\xi, \frac{k^{2}}{4}\right)\textrm{se}_{n}\left(i\xi', \frac{k^{2}}{4}\right)\textrm{se}_{n}\left(\eta, \frac{k^{2}}{4}\right)\textrm{se}_{n}\left(\eta', \frac{k^{2}}{4}\right)$$
\begin{equation}\left.\hskip0.8cm\times \textrm{se}_{n}\left(i\xi_{0}, \frac{k^{2}}{4}\right)\partial_{\xi_{0}}\textrm{Ms}_{n}^{(3)}\left(\xi_{0}, \frac{k}{2}\right)\right].\end{equation}
\vskip0.2cm
Thus, from the last sum it clearly follows that $\mathcal{I}$ is a symmetric function in the variables $(\xi, \eta)$ and $(\xi', \eta')$.

\begin{remark}

Observe that the expansion (3.6) is analogous to the expansion (1.4) of the integral $\mathcal{I}$ from the circular case to the elliptical case.

\end{remark}

\section{Appendix}

\begin{lem}

Let $(\xi_{0}, \eta_{0})$ and $(\xi_{1}, \eta_{1})$ be two points in $\Bbb R^{2}$ and let $\xi_{2}$ be a positive real number such that $|\xi_{0}| < \xi_{2}$. Then, we have the following identity

$$\hskip-10.5cm u_{k}(\xi_{0}, \eta_{0}, \xi_{1}, \eta_{1})$$
$$ \hskip-6cm = \frac{1}{4}\int_{-\pi}^{\pi}\left( v_{k}(\xi_{2}, \eta, \xi_{0}, \eta_{0})\partial_{1}u_{k}(\xi_{2}, \eta, \xi_{1}, \eta_{1})\right.$$
$$ \hskip-3.5cm - \left.u_{k}(\xi_{2}, \eta, \xi_{1}, \eta_{1})\partial_{1}v_{k}(\xi_{2}, \eta, \xi_{0}, \eta_{0})\right)d\eta.$$

\end{lem}

\begin{proof}

The radial Bessel function of the second kind $Y_{0}(k|\cdot|)$ is a fundamental solution for the Helmholtz operator $\triangle_{x} + k^{2}$ in $\Bbb R^{2}$. That is, if
\begin{equation}\hskip-7cm G(x) = \int_{\Bbb R^{2}}Y_{0}(k|x - y|)g(y)dy,\end{equation}
then
\vskip-0.2cm
$$\hskip-7.7cm\left(\triangle_{x} + k^{2}\right)G(x) = -4g(x)$$
where $g$ is any continuous function in $\Bbb R^{2}$ with compact support (see \cite[page 3]{11}). In elliptical coordinates the Helmholtz operator has the form
\vskip-0.2cm
$$\hskip-3.5cm\triangle_{x} + k^{2} = \frac{1}{\cosh\xi^{2} - \cos^{2}\eta}\left(\frac{\partial^{2}}{\partial\xi^{2}} + \frac{\partial^{2}}{\partial\eta^{2}}\right) + k^{2}.$$
Hence, using the change of variables
\vskip-0.2cm
$$\hskip-4.3cm y = x(\xi', \eta'), dy = \left(\cosh^{2}\xi' - \cos^{2}\eta'\right)d\eta' d\xi'$$
in equation (4.1) and then taking the Helmholtz operator we obtain that
$$ \hskip-1cm - 4g\left(x(\xi, \eta)\right) = \left(\frac{1}{\cosh\xi^{2} - \cos^{2}\eta}\left(\frac{\partial^{2}}{\partial\xi^{2}} + \frac{\partial^{2}}{\partial\eta^{2}}\right) + k^{2}\right)G\left(x(\xi, \eta)\right)$$
$$ \hskip-5cm = \left(\frac{1}{\cosh\xi^{2} - \cos^{2}\eta}\left(\frac{\partial^{2}}{\partial\xi^{2}} + \frac{\partial^{2}}{\partial\eta^{2}}\right) + k^{2}\right)$$ $$\hskip-2cm\times\int_{0}^{\infty}\int_{-\pi}^{\pi}v_{k}(\xi, \eta, \xi', \eta')g(x(\xi', \eta'))\left(\cosh^{2}\xi' - \cos^{2}\eta'\right)d\eta' d\xi'$$
where we used equation (1.13) relating the function $v_{k}$ with $Y_{0}$. This implies that
\vskip-0.2cm
$$ \hskip-5.3cm = \left(\frac{\partial^{2}}{\partial\xi^{2}} + \frac{\partial^{2}}{\partial\eta^{2}} + k^{2}\left(\cosh\xi^{2} - \cos^{2}\eta\right)\right)$$
$$\hskip-2cm\times\int_{0}^{\infty}\int_{-\pi}^{\pi}v_{k}(\xi, \eta, \xi', \eta')g(x(\xi', \eta'))\left(\cosh^{2}\xi' - \cos^{2}\eta'\right)d\eta' d\xi'$$
$$ \hskip-6cm = -4g\left(x(\xi, \eta)\right)\left(\cosh\xi^{2} - \cos^{2}\eta\right).$$
Thus, $v_{k}(\xi, \eta, \xi', \eta')$ is a fundamental solution, at the point $(\xi', \eta')$, for the differential operator
\vskip-0.2cm
$$\hskip-4.7cm L_{\xi, \eta} = \frac{\partial^{2}}{\partial\xi^{2}} + \frac{\partial^{2}}{\partial\eta^{2}} + k^{2}\left(\cosh\xi^{2} - \cos^{2}\eta\right).$$
Let $(\xi_{0}, \eta_{0}), (\xi_{1}, \eta_{1})$ be two arbitrary points in $\Bbb R^{2}$ such that $|\eta_{0}| < \pi$. Let us use Green's Theorem for the functions $u_{k}\left(\partial_{2}v_{k}\right) - \left(\partial_{2}u_{k}\right)v_{k}$ and $\left(\partial_{1}u_{k}\right)v_{k} - u_{k}\left(\partial_{1}v_{k}\right)$, where
\vskip-0.2cm
$$\hskip-4.5cm u_{k} = u_{k}\left(\cdot, \cdot, \xi_{1}, \eta_{1}\right), v_{k} = v_{k}\left(\cdot, \cdot, \xi_{0}, \eta_{0}\right)$$
on the rectangle
\vskip-0.2cm
$$\hskip-5cm R = \left\{(\xi, \eta):|\xi| < \xi_{2}, - \pi\leq \eta\leq\pi\right\}$$
where we assume that $|\xi_{0}| < \xi_{2}$. Using the fact that $u_{k}$ is in the kernel of the operator $L_{\xi, \eta}$ (because of formula (1.12) and the fact that $J_{0}(k|\cdot|)$ is in the kernel of the Helmholtz operator $\triangle_{x} + k^{2}$) and fact that the rectangle $R$ contains the singularity points $\pm(\xi_{0}, \eta_{0})$ of $v_{k}$ we have
\vskip-0.2cm
$$ \hskip-3cm \oint_{\partial R}\left[u_{k}\left(\partial_{2}v_{k}\right) - \left(\partial_{2}u_{k}\right)v_{k}\right]d\xi + \left[\left(\partial_{1}u_{k}\right)v_{k} - u_{k}\left(\partial_{1}v_{k}\right)\right]d\eta$$
$$ \hskip-0.5cm = \int_{R}\left[\frac{\partial}{\partial\xi}\left(\frac{\partial u_{k}(\xi, \eta, \xi_{1}, \eta_{1})}{\partial\xi}v_{k}(\xi, \eta, \xi_{0}, \eta_{0}) - u_{k}(\xi, \eta, \xi_{1}, \eta_{1})\frac{\partial v_{k}(\xi, \eta, \xi_{0}, \eta_{0})}{\partial\xi}\right)\right.$$ $$\left.- \frac{\partial}{\partial\eta}\left(u_{k}(\xi, \eta, \xi_{1}, \eta_{1})\frac{\partial v_{k}(\xi, \eta, \xi_{0}, \eta_{0})}{\partial\eta} - \frac{\partial u_{k}(\xi, \eta, \xi_{1}, \eta_{1})}{\partial\eta}v_{k}(\xi, \eta, \xi_{0}, \eta_{0})\right)\right]d\xi d\eta$$
$$ \hskip-5cm = \int_{R}\left[\left(\frac{\partial^{2}}{\partial\xi^{2}} + \frac{\partial^{2}}{\partial\eta^{2}}\right)u_{k}(\xi, \eta, \xi_{1}, \eta_{1})v_{k}(\xi, \eta, \xi_{0}, \eta_{0})\right.$$ $$\hskip-4cm\left. - \left(\frac{\partial^{2}}{\partial\xi^{2}} + \frac{\partial^{2}}{\partial\eta^{2}}\right)v_{k}(\xi, \eta, \xi_{0}, \eta_{0})u_{k}(\xi, \eta, \xi_{1}, \eta_{1})\right]d\xi d\eta$$
$$ = -\int_{R}u_{k}(\xi, \eta, \xi_{1}, \eta_{1})\left(\frac{\partial^{2}}{\partial\xi^{2}} + \frac{\partial^{2}}{\partial\eta^{2}} + k^{2}\left(\cosh\xi^{2} - \cos^{2}\eta\right)\right)v_{k}(\xi, \eta, \xi_{0}, \eta_{0})d\xi d\eta$$
\begin{equation}\hskip-2.75cm = 4u_{k}(\xi_{0}, \eta_{0}, \xi_{1}, \eta_{1}) + 4u_{k}(-\xi_{0}, -\eta_{0}, \xi_{1}, \eta_{1}) = 8u_{k}(\xi_{0}, \eta_{0}, \xi_{1}, \eta_{1}).\end{equation}
On the other hand, we can write more explicitly
\vskip-0.2cm
$$ \hskip-2.5cm \oint_{\partial R}\left[u_{k}\left(\partial_{2}v_{k}\right) - \left(\partial_{2}u_{k}\right)v_{k}\right]d\xi + \left[\left(\partial_{1}u_{k}\right)v_{k} - u_{k}\left(\partial_{1}v_{k}\right)\right]d\eta$$
$$ \hskip-5.5cm = \int_{-\xi_{2}}^{\xi_{2}}\left(\partial_{2} v_{k}(\xi, - \pi, \xi_{0}, \eta_{0})u_{k}(\xi, - \pi, \xi_{1}, \eta_{1})\right.$$
$$ \hskip-4cm \left.- \partial_{2}u_{k}(\xi, - \pi, \xi_{1}, \eta_{1})v_{k}(\xi, - \pi, \xi_{0}, \eta_{0})\right)d\xi$$
$$ \hskip-5.8cm + \int_{-\pi}^{\pi}\left(v_{k}(\xi_{2}, \eta, \xi_{0}, \eta_{0})\partial_{1} u_{k}(\xi_{2}, \eta, \xi_{1}, \eta_{1})\right.$$
$$ \hskip-4cm \left.- u_{k}(\xi_{2}, \eta, \xi_{1}, \eta_{1})\partial_{1} v_{k}(\xi_{2}, \eta, \xi_{0}, \eta_{0})\right)d\eta$$
$$ \hskip-5.5cm - \int_{-\xi_{2}}^{\xi_{2}}\left(\partial_{2} v_{k}(\xi, \pi, \xi_{0}, \eta_{0})u_{k}(\xi, \pi, \xi_{1}, \eta_{1})\right.$$
$$ \hskip-4cm \left.- \partial_{2}u_{k}(\xi, \pi, \xi_{1}, \eta_{1})v_{k}(\xi, \pi, \xi_{0}, \eta_{0})\right)d\xi$$
$$ \hskip-4.7cm - \int_{-\pi}^{\pi}\left(v_{k}(-\xi_{2}, \eta, \xi_{0}, \eta_{0})\partial_{1} u_{k}(-\xi_{2}, \eta, \xi_{1}, \eta_{1})\right.$$
$$ \hskip-2.5cm \left.- u_{k}(-\xi_{2}, \eta, \xi_{1}, \eta_{1})\partial_{1} v_{k}(-\xi_{2}, \eta, \xi_{0}, \eta_{0})\right)d\eta$$
\begin{equation} \hskip-9cm = \mathcal{I}_{1} + \mathcal{I}_{2} + \mathcal{I}_{3} + \mathcal{I}_{4}.\end{equation}
Observe that since $v_{k}$ and $u_{k}$ are $2\pi$ periodic then it follows that the integrals $\mathcal{I}_{1}$ and $\mathcal{I}_{3}$ cancel each other. Now, we claim that $\mathcal{I}_{2} = \mathcal{I}_{4}$. Indeed, since both $u_{k}(\cdot, \cdot, \xi_{1}, \eta_{1})$ and $v_{k}(\cdot, \cdot, \xi_{0}, \eta_{0})$ are even functions it follows that
$$\hskip-5cm v_{k}(-\xi, \eta, \xi_{0}, \eta_{0}) = v_{k}(\xi, -\eta, \xi_{0}, \eta_{0}),$$
$$\hskip-5cm u_{k}(-\xi, \eta, \xi_{0}, \eta_{0}) = u_{k}(\xi, -\eta, \xi_{0}, \eta_{0}).$$
Hence, taking the derivative with respect to $\xi$ on both sides of the last two equations we obtain that
\vskip-0.2cm
$$\hskip-4cm -\partial_{1}v_{k}(-\xi, \eta, \xi_{0}, \eta_{0}) = \partial_{1}v_{k}(\xi, -\eta, \xi_{0}, \eta_{0}),$$
$$\hskip-4cm -\partial_{1}u_{k}(-\xi, \eta, \xi_{0}, \eta_{0}) = \partial_{1}u_{k}(\xi, -\eta, \xi_{0}, \eta_{0}).$$
Hence, we have
\vskip-0.2cm
$$\hskip-14cm \mathcal{I}_{4} $$
$$ \hskip0.2cm = - \int_{-\pi}^{\pi}\left(v_{k}(-\xi_{2}, \eta, \xi_{0}, \eta_{0})\partial_{1}u_{k}(-\xi_{2}, \eta, \xi_{1}, \eta_{1})
- u_{k}(-\xi_{2}, \eta, \xi_{1}, \eta_{1})\partial_{1}v_{k}(-\xi_{2}, \eta, \xi_{0}, \eta_{0})\right)d\eta$$
$$ \hskip0.2cm = \int_{-\pi}^{\pi}\left(v_{k}(\xi_{2}, -\eta, \xi_{0}, \eta_{0})\partial_{1}u_{k}(\xi_{2}, -\eta, \xi_{1}, \eta_{1})
- u_{k}(\xi_{2}, -\eta, \xi_{1}, \eta_{1})\partial_{1}v_{k}(\xi_{2}, -\eta, \xi_{0}, \eta_{0})\right)d\eta$$
Thus, by making the change of variables $\eta\mapsto -\eta$ in the last integral we see that both integrals $\mathcal{I}_{2}$ and $\mathcal{I}_{4}$ coincide. Combining equations (4.2) and (4.3) we obtain Lemma 4.1.

\end{proof}

\begin{lem}

Assume that $|\xi'| < \xi$. Then, the function $v_{k}$ has the following expansion

$$ \hskip-1cm v_{k}(\xi', \eta', \xi, \eta) = 2\pi i\mathrm{Mc}_{0}^{(1)}\left(\xi', \frac{k}{2}\right)\mathrm{ce}_{0}\left(\eta', \frac{k^{2}}{4}\right)\mathrm{Mc}_{0}^{(3)}\left(\xi, \frac{k}{2}\right)\mathrm{ce}_{0}\left(\eta, \frac{k^{2}}{4}\right)$$
$$ \hskip0.8cm + 2\pi i\sum_{n = 1}^{\infty}\left[\mathrm{Mc}_{n}^{(1)}\left(\xi', \frac{k}{2}\right)\mathrm{ce}_{n}\left(\eta', \frac{k^{2}}{4}\right)\mathrm{Mc}_{n}^{(3)}\left(\xi, \frac{k}{2}\right)\mathrm{ce}_{n}\left(\eta, \frac{k^{2}}{4}\right)\right.$$
$$ \hskip2.3cm \left. + \mathrm{Ms}_{n}^{(1)}\left(\xi', \frac{k}{2}\right)\mathrm{se}_{n}\left(\eta', \frac{k^{2}}{4}\right)\mathrm{Ms}_{n}^{(3)}\left(\xi, \frac{k}{2}\right)\mathrm{se}_{n}\left(\eta, \frac{k^{2}}{4}\right)\right]$$

into eigenfunctions in elliptical coordinates.

\end{lem}

\begin{proof}

The proof of Lemma 4.2 is based on the proof of Theorem 5.2 in \cite{7}. There, the authors derive the expansion into eigenfunctions in elliptical coordinates of the modified Bessel function of the second kind $K_{0}$. Here, we will use a slight modification of the proof introduced in \cite{7} in order to adjust it to the Bessel function of the second kind $Y_{0}$.

Since $\xi'\neq\pm\xi$ then for every fixed $\xi, \xi'$ and $\eta'$ the function $f_{\xi, \xi',\eta'}(z) = v_{k}(\xi, z, \xi', \eta')$ can be extended, in an open strip containing $\Bbb R$, into a $2\pi$ periodic analytic function. Thus, it has an expansion in the form of (1.9):
$$\hskip-6.5cm v_{k}(\xi, \eta, \xi', \eta') = \alpha_{0}(\xi, \xi', \eta')\textrm{ce}_{0}\left(\eta, \frac{k^{2}}{4}\right)$$
$$\hskip2cm + \sum_{n = 1}^{\infty}\left(\alpha_{n}(\xi, \xi', \eta')\textrm{ce}_{n}\left(\eta, \frac{k^{2}}{4}\right) + \beta_{n}(\xi, \xi', \eta')\textrm{se}_{n}\left(\eta, \frac{k^{2}}{4}\right)\right)$$
where
\vskip-0.2cm
$$\hskip-4.7cm\alpha_{n}(\xi, \xi', \eta') = \frac{1}{\pi}\int_{-\pi}^{\pi}v_{k}(\xi, \eta, \xi', \eta')\textrm{ce}_{n}\left(\eta, \frac{k^{2}}{4}\right)d\eta,$$
$$\hskip-4.7cm\beta_{n}(\xi, \xi', \eta') = \frac{1}{\pi}\int_{-\pi}^{\pi}v_{k}(\xi, \eta, \xi', \eta')\textrm{se}_{n}\left(\eta, \frac{k^{2}}{4}\right)d\eta.$$
Hence, our aim is to find the explicit form of the functions $\alpha_{n}$ and $\beta_{n}$.

Now, we are going to use again Lemma 4.1. This lemma was formulated for the function $u_{k}$, however it is clear from its proof that it is true when $u_{k}$ is replaced with any function $U(\xi, \eta)$ which is in the kernel of the operator $L_{\xi, \eta}$, even with respect to the variable $(\xi,\eta)$ (i.e., $U(\xi, \eta) = U(-\xi, -\eta)$) and which is $2\pi$ periodic with respect to the variable $\eta$. Hence, choosing $U(\xi, \eta) = u(\xi)\textrm{ce}_{n}\left(\eta, \frac{k^{2}}{4}\right)$ where $u$ is an even function, which is a solution to (1.7), we have
\vskip0.2cm
$$\hskip-11cm u(\xi')\textrm{ce}_{n}\left(\eta', \frac{k^{2}}{4}\right)$$
$$ = \frac{1}{4}\int_{-\pi}^{\pi}\left(u'(\xi)\textrm{ce}_{n}\left(\eta, \frac{k^{2}}{4}\right)v_{k}(\xi, \eta, \xi', \eta') - \partial_{1} v_{k}(\xi, \eta, \xi', \eta')u(\xi)\textrm{ce}_{n}\left(\eta, \frac{k^{2}}{4}\right)\right)d\eta$$
where the last identity is true in case where $|\xi'| < \xi$. Hence, if we define the following function
\vskip-0.2cm
\begin{equation}\hskip-6.3cm f(\xi) = \frac{1}{4}\int_{-\pi}^{\pi}v_{k}(\xi, \eta, \xi', \eta')\textrm{ce}_{n}\left(\eta, \frac{k^{2}}{4}\right)d\eta,\end{equation}
then we have
\vskip-0.2cm
$$\hskip-6cm u(\xi')\textrm{ce}_{n}\left(\eta', \frac{k^{2}}{4}\right) = u'(\xi)f(\xi) - u(\xi)f'(\xi).$$
Let us choose $u(\xi) = \textrm{Mc}_{n}^{(1)}\left(\xi, \frac{k}{2}\right)$, then we have
$$\hskip-1cm\textrm{Mc}_{n}^{(1)}\left(\xi', \frac{k}{2}\right)\textrm{ce}_{n}\left(\eta', \frac{k^{2}}{4}\right) = \partial_{\xi}\textrm{Mc}_{n}^{(1)}\left(\xi, \frac{k}{2}\right)f(\xi) - \textrm{Mc}_{n}^{(1)}\left(\xi, \frac{k}{2}\right)f'(\xi)$$
\begin{equation}\hskip0.4cm = W\left[f, \textrm{Mc}_{n}^{(1)}\left(\cdot, \frac{k}{2}\right)\right]\end{equation}
where $W$ denotes the wronskian. Thus, the wronskian of $f$ and $\textrm{Mc}_{n}^{(1)}$ is equal to a constant which does not depend on $\xi$. However, is was proved in \cite[page 668, formula 28.20.1]{4} that the wronskian of $\textrm{Mc}_{n}^{(1)}$ and $\textrm{Mc}_{n}^{(3)}$ is also equal to a constant. Hence, there exist constants $c_{1}$ and $c_{2}$ which do not depend on $\xi$ such that
$$\hskip-4.5cm f(\xi) = c_{1}\textrm{Mc}_{n}^{(3)}\left(\xi,\frac{k}{2}\right) + c_{2}\textrm{Mc}_{n}^{(1)}\left(\xi,\frac{k}{2}\right), \xi > |\xi'|.$$
However, we claim that $c_{2} = 0$. Indeed, taking the analytic extensions of the functions $f, \textrm{Mc}_{n}^{(3)}$ and $\textrm{Mc}_{n}^{(1)}$ in a narrow strip which contains the ray $[\xi',\infty)$ we see, from equations (1.10) and (4.4), that both $f$ and $\textrm{Mc}_{n}^{(3)}$ have the same asymptotic expansion at infinity (observe that $f(z)\approx Y_{0}(k\cosh z)$) while $\textrm{Mc}_{n}^{(1)}$ behaves like  (see equation (1.11))$e^{\left|\Im\left(k\cosh z\right)\right|}O\left(\left(\cosh z\right)^{-\frac{3}{2}}\right)$.  This implies that $c_{2} = 0$ and hence
\begin{equation}\hskip-6.5cm f(\xi) = c_{1}(\xi', \eta')\textrm{Mc}_{n}^{(3)}\left(\xi,\frac{k}{2}\right), \xi > |\xi'|.\end{equation}
In order to find the constant $c_{1}$ we use the fact that the wronskian of $\textrm{Mc}_{n}^{(1)}$ and $\textrm{Mc}_{n}^{(3)}$ is equal to $- 2i / \pi$ and equation (4.5) in order to obtain that
$$\hskip-7cm c_{1} = -\frac{\pi}{2i}\textrm{Mc}_{n}^{(1)}\left(\xi', \frac{k}{2}\right)\textrm{ce}_{n}\left(\eta', \frac{k^{2}}{4}\right).$$
Hence, using equations (4.4) and (4.6) we finally obtain that

$$\hskip-8cm\int_{-\pi}^{\pi}v_{k}(\xi, \eta, \xi', \eta')\textrm{ce}_{n}\left(\eta, \frac{k^{2}}{4}\right)d\eta $$ $$\hskip-3cm = 2\pi i\textrm{Mc}_{n}^{(1)}\left(\xi', \frac{k}{2}\right)\textrm{ce}_{n}\left(\eta', \frac{k^{2}}{4}\right)\textrm{Mc}_{n}^{(3)}\left(\xi,\frac{k}{2}\right), |\xi'| < \xi.$$

In the exact same way we can show that
\vskip-0.2cm
$$\hskip-8cm\int_{-\pi}^{\pi}v_{k}(\xi, \eta, \xi', \eta')\textrm{se}_{n}\left(\eta, \frac{k^{2}}{4}\right)d\eta $$ $$\hskip-3cm = 2\pi i\textrm{Ms}_{n}^{(1)}\left(\xi', \frac{k}{2}\right)\textrm{se}_{n}\left(\eta', \frac{k^{2}}{4}\right)\textrm{Ms}_{n}^{(3)}\left(\xi,\frac{k}{2}\right), |\xi'| < \xi.$$

This proves Lemma 4.2.

\end{proof}

\begin{lem}

Let $f$ be a continuous function defined on $\Bbb R^{2}$ with compact support, then $f$ has the following representation
$$\hskip-6.5cm f(x) = \frac{1}{2\pi}\int_{0}^{\infty}\int_{\Bbb R^{2}}f(y)J_{0}(k|x - y|)dykdk.$$

\end{lem}

\begin{proof}

From the inversion formula for the Fourier transform in $\Bbb R^{2}$ we obtain that
$$\hskip-8.5cm f(x) = \frac{1}{2\pi}\int_{\Bbb R^{2}}\mathcal{F}(f)(y)e^{ix\cdot y}dy$$
where $\mathcal{F}(f)$ denotes the Fourier transform of $f$. Using the definition of the Fourier transform we have
$$\hskip-1.5cm f(x) = \frac{1}{4\pi^{2}}\int_{\Bbb R^{2}}\int_{\Bbb R^{2}}f(z)e^{-iz\cdot y}dze^{ix\cdot y}dy = \frac{1}{4\pi^{2}}\int_{\Bbb R^{2}}\int_{\Bbb R^{2}}f(z)e^{i(x - z)\cdot y}dydz$$
$$ = \left[y = ke^{i\theta}, dy = kd\theta dk\right] = \frac{1}{4\pi^{2}}\int_{\Bbb R^{2}}\int_{0}^{\infty}\int_{-\pi}^{\pi}f(z)e^{ik(x - z)\cdot e^{i\theta}}d\theta kdkdz$$
$$\hskip-5.7cm = \frac{1}{2\pi}\int_{0}^{\infty}\int_{\Bbb R^{2}}f(z)J(k|x - z|)dzkdk$$
where in the last passage we used the following formula
$$\hskip-9.5cm J_{0}(t) = \frac{1}{2\pi}\int_{-\pi}^{\pi}e^{it\cos\theta}d\theta $$
for the Bessel function of the first kind. This proves Lemma 4.3.

\end{proof}

\end{document}